\documentclass[a4paper]{amsart}

\usepackage{amsmath,amssymb,amsthm}
\usepackage{eucal}
\usepackage[all]{xy}
\usepackage[dvipdfmx,colorlinks=true, backref=page]{hyperref}

\newtheorem{theorem}{Theorem}[section]
\newtheorem{proposition}[theorem]{Proposition}
\newtheorem{lemma}[theorem]{Lemma}
\newtheorem{corollary}[theorem]{Corollary}

\theoremstyle{definition}
\newtheorem{definition}[theorem]{Definition}
\newtheorem{example}[theorem]{Example}
\newtheorem{conjecture}[theorem]{Conjecture}

\theoremstyle{remark}
\newtheorem{remark}[theorem]{Remark}

\numberwithin{equation}{section}

\newcommand{\Z}{\mathbb{Z}}
\newcommand{\Q}{\mathbb{Q}}

\newcommand{\C}{\mathbb{C}}
\newcommand{\spin}{\operatorname{spin}}

\SelectTips{cm}{}

\title[Upper bounds for virtual dimensions]{Upper bounds for virtual dimensions of Seiberg-Witten moduli spaces}

\author[T. Kato]{Tsuyoshi Kato}
\address{Department of Mathematics, Kyoto University, Kyoto, 606-8502, Japan}
\email{tkato@math.kyoto-u.ac.jp}

\author[D. Kishimoto]{Daisuke Kishimoto}
\address{Faculty of Mathematics, Kyushu University, Fukuoka 819-0395, Japan}
\email{kishimoto@math.kyushu-u.ac.jp}

\author[N. Nakamura]{Nobuhiro Nakamura}
\address{Integrated Center for Science and Humanities, Fukushima Medical University, 
1 Hikariga-oka, Fukushima City 960-1295, Japan}
\email{nnaka@fmu.ac.jp}

\author[K. Yasui]{Kouichi Yasui}
\address{Department of Pure and Applied Mathematics, Graduate School of Information Science and Technology, Osaka University,  Suita,  565-0871, Japan}
\email{kyasui@ist.osaka-u.ac.jp}



\begin{document}

\maketitle

\begin{abstract}
 Given a closed four-manifold with $b_1=0$ and a prime number $p$, we prove that for any mod $p^r$ basic class, the virtual dimension of the Seiberg-Witten moduli space is bounded above by $2r(p-1)-2$ under some conditions on $r$ and $b_2^+$. As an application, we obtain adjunction inequalities for embedded surfaces with negative self-intersection number.
\end{abstract}


\section{Introduction}\label{introduction}

The Seiberg-Witten invariant of a smooth four-manifold has been playing a fundamental role in the study of four-manifolds, and has been a rich source of ideas and applications. The basic ingredient in the construction of the invariant is the moduli space of solutions of the Seiberg-Witten equation, which is defined for each $\spin^c$ structure on a four-manifold. The virtual dimension of the moduli space is of particular importance, and there is a fundamental conjecture on it. Let us recall it here. Let $X$ be a closed, connected, oriented and smooth four-manifold, and let $\mathfrak{s}$ be a $\spin^c$ structure on $X$. We will omit a four-manifold $X$ in the notation if it is clear from the context. We say that (the isomorphism class of) $\mathfrak{s}$ is a \emph{basic class} if the Seiberg-Witten invariant $SW(\mathfrak{s})\ne 0$. Let $d(\mathfrak{s})$ denote the virtual dimension of the moduli space corresponding to $\mathfrak{s}$. We define that $X$ is of \emph{simple type} if $d(\mathfrak{s})=0$ whenever $\mathfrak{s}$ is a basic class. Now we state the so-called \emph{simple type conjecture} (see \cite[Conjecture 1.6.2]{KM}).

\begin{conjecture}\label{conjecture simple type}
  Every closed, connected, oriented and smooth four-manifold with $b_2^+\ge 2$ is of simple type.
\end{conjecture}

The simple type conjecture was originally posed in connection to Witten's conjecture on the relationship between the Donaldson and the Seiberg-Witten invariants for four-manifolds of simple type.
By a partial solution due to Feehan and Leness \cite{FL15}, the simple type conjecture implies that the relation of these invariants holds under mild topological conditions only.
%

The simple type conjecture trivially holds in the case where $b_2^+-b_1$ is even, since the Seiberg-Witten invariant always vanishes. In the case that $b_2^+-b_1$ is odd, the simple type conjecture has been verified for all symplectic four-manifolds \cite{Ta} and also for other very large families of four-manifolds. However, so far, there is no result without demanding a condition on a smooth structure,
except for the following work.
We say that $\mathfrak{s}$ is a \emph{mod $q$ basic class} if $SW(\mathfrak{s})\not\equiv 0\mod q$. So we can consider the mod $q$ analogue of the simple type conjecture, and recently, Kato, Nakamura and Yasui \cite{KNY} solved the mod 2 analogue under a mild condition on the cohomology ring, which depends only on the underlying topological structure.

In this paper, under a simple topological condition, we give an upper bound for the virtual dimension, giving a new approach to the simple type conjecture. Now we state our main theorem.
For a prime $p$ and integers $k,t$ with $k\equiv t-2 \mod p$, we define an integer $a(k,t)$ by
\begin{equation}
  \label{a(k,t)}
  k+(t-3)(p-1)=a(k,t)p+1.
\end{equation}
For an integer $r$, let $1\le r_p\le p-1$ be the integer such that $r_p=1$ for $r\equiv 0\mod p$ and $r\equiv r_p\mod p$ for $r\not\equiv 0\mod p$.


\begin{theorem}
  \label{main2}
 Let $p$ be a prime, and suppose that $b_2^+$ is odd with $b_2^+> 2$ and $b_1=0$. Let $\mathfrak{s}$ be a $\spin^c$ structure on $X$, and set $k=(b_2^+-1)/2$.  Take an integer $r$ satisfying $1\le r<p(p-1)$. If $\mathfrak{s}$ is a mod $p^r$ basic class, then
  \[
    d(\mathfrak{s})\le 2r(p-1)-2
  \]
  whenever $k,r$ satisfy the following conditions:
  \begin{enumerate}
    \item $k\not\equiv 0,1,\ldots,r_p-1\mod p$;

    \item 
under the above condition,   if an integer $t$ satisfies $t-2\equiv k \mod p$ and $3\le t\le r$, then  $a(k,t)$ satisfies
            \begin{alignat*}{3}
          &3a(k,t)+5\not\equiv 0&&\mod p&&(t\equiv 0\mod p\text{ and }t\ge p>3)\\
          &a(k,t)+2\not\equiv 0&&\mod p&&(t\equiv 1\mod p\text{ and }t>p)\\
          &3a(k,t)+4\not\equiv 0&&\mod p&&(t\equiv 3\mod p)\\
          &(2t-3)a(k,t)+3t-5\not\equiv 0&&\mod p\qquad&&(t \equiv 4, 5, \cdots , p-1\mod p).
        \end{alignat*}
  \end{enumerate}
\end{theorem}

To prove Theorem \ref{main2}, we will employ the Bauer-Furuta invariant, which is a lift of the Seiberg-Witten invariant to the stable cohomotopy group. This will enable us to deduce the divisibility of the Seiberg-Witten invariant from a property of the $p$-localized cohomotopy groups of complex projective spaces, that will be proved by computing Toda brackets based on the $p$-local cell structure of a complex projective space. Using techniques in hard homotopy theory such as Toda brackets is new in the study of Seiberg-Witten invariant, though it is standard in algebraic topology.

\begin{corollary}
  \label{main3}
  Let $p$ be a prime, and suppose that $b_2^+$ is odd with $b_2^+>2$ and $b_1=0$. If $k=(b_2^+-1)/2$ and an integer $r$ satisfies $k(k-1)\cdots(k-r+1)\not\equiv 0\mod p$ and $\mathfrak{s}$ is a mod $p^r$ basic class, then
  \[
    d(\mathfrak{s})\le 2r(p-1)-2.
  \]
\end{corollary}

\begin{proof}
 If $k\not\equiv 0,1,\ldots,r_p-1\mod p$ exists, then there is no integer $t$ satisfying $t-2\equiv k\mod p$ and $3\le t\le r$. Then the statement follows from Theorem \ref{main2}.
\end{proof}

\begin{corollary}
  \label{main}
  Let $p$ be a prime, and suppose that $b_2^+$ is odd with $b_2^+>2$ and $b_1=0$.
  If $(b_2^+-1)/2\not\equiv 0\mod p$ and $\mathfrak{s}$ is a mod $p$ basic class, then
  \[
    d(\mathfrak{s})\le 2p-4.
  \]
\end{corollary}

\begin{proof}
Apply Theorem \ref{main2} for $r=1$.
Note that the condition
of $k^2 \not\equiv 0$ $\mod p$
is equivalent to the one of $k \not\equiv 0$ $\mod p$.
\end{proof}

We can deduce Corollaries \ref{main3} and \ref{main} from the result of Bauer and Furuta \cite[Theorem 3.7]{BF} by a purely algebraic argument, as in Section \ref{inexplicit upper bound}. 
The authors thank the referee of the earlier  draft for  pointing out this algebraic argument for Corollary \ref{main}.
On the other hand, Theorem \ref{main2} gives an upper bound better than the one deduced from \cite[Theorem 3.7]{BF}. For instance, if $k=1$ and $r=p$, then by Theorem \ref{main2}, we get
\[
  d(\mathfrak{s})\le 2r(p-1)-2
\]
for $SW(\mathfrak{s})\not\equiv 0\mod p^r$, whereas we only can deduce a rather weaker inequality
\[
  d(\mathfrak{s})\le 2p^r-4
\]
from Bauer-Furuta \cite[Theorem 3.7]{BF}. See Section \ref{inexplicit upper bound}.

A straightforward corollary of Corollary \ref{main}
below gives an upper bound for any basic class, not a mod $p$ basic class.

\begin{corollary}
  \label{uniform bound 2}
  Suppose that $b_2^+$ is odd with $b_2^+>2$ and $b_1=0$. For a basic class $\mathfrak{s}$, let $p$ be the least prime not dividing $SW(\mathfrak{s})$ and satisfying $(b_2^+-1)/2\not\equiv 0\mod p$. Then
  \[
    d(\mathfrak{s})\le 2p-4.
  \]
\end{corollary}

From this corollary, we can derive a coarse but more concrete upper bound.

\begin{corollary}
  \label{non-prime bound}
  Suppose that $b_2^+$ is odd with $b_2^+>2$ and $b_1=0$. Then every basic class $\mathfrak{s}$ satisfies
  \[
    d(\mathfrak{s})\le \max\{2|SW(\mathfrak{s})|-6,\, b_2^+-7,\, 10\}.
  \]
\end{corollary}

\begin{remark}
  \label{remark upper bound}
  Theorem \ref{upper bound S} below shows that we can get a sharper upper bound than Corollary \ref{non-prime bound} if we assume either $|SW(\mathfrak{s})|$ or $b_2^+$ is large enough. For example (Example \ref{n large}), if $\max\{|SW(\mathfrak{s})|,\,(b_2^+-1)/2\}\ge 44$, then
  \[
    d(\mathfrak{s})\le\max\left\{\frac{4}{3}|SW(\mathfrak{s})|-4,\,\frac{2}{3}b_2^+-\frac{14}{3}\right\}.
  \]
\end{remark}




Let us consider a small prime $p$. Clearly, if $p$ is small, then the upper bound in Corollary \ref{main} imposes a strong constraint on the virtual dimension $d(\mathfrak{s})$. Here, we consider the $p=2,3$ cases. For $p=2$, $(b_2^+-1)/2\not\equiv 0\mod 2$ is equivalent to $b_2^+\equiv 3\mod 4$ in Corollary \ref{main} because $SW(\mathfrak{s})=0$ for $b_2^+$ even. Then we obtain that if $b_2^+\ge 2$, $b_1=0$ and $b_2^+\equiv 3\mod 4$, then $d(\mathfrak{s})=0$ for every mod 2 basic class $\mathfrak{s}$. This is the special case $b_1=0$ of \cite[Corollary 1.4]{KNY}, the above mentioned solution to the mod 2 analogue of the simple type conjecture. Our proof is very different from theirs. Indeed, their proof relies on connected sum formulae for the Bauer-Furuta invariant \cite{B,IS} and thus does not give upper bounds for virtual dimensions in the $p\geq 3$ case. For $p$ odd, $(b_2^+-1)/2\not\equiv 0\mod p$ is equivalent to $b_2^+\not\equiv 1\mod p$. Then for $p=3$, we have:

\begin{corollary}
  Suppose that $b_2^+$ is odd with $b_2^+>2$ and $b_1=0$. If $b_2^+\not\equiv 1\mod 3$ and $\mathfrak{s}$ is a mod 3 basic class, then
  \[
    d(\mathfrak{s})=0\text{ or }2.
  \]
\end{corollary}




We turn to an application of Corollary \ref{main}. A typical application of an affirmative solution to the simple type conjecture is an adjunction inequality for embedded surfaces with negative self-intersection number in the case $b_1=0$ \cite{OS1}. Here, we prove an adjunction inequality from Corollary \ref{main}, instead of assuming manifolds being of simple type. For a second homology class $\alpha$ of a four-manifold, an adjunction inequality gives a lower bound for the genus of a smoothly embedded closed oriented surface representing $\alpha$, and adjunction inequalities have various powerful applications to four-dimensional topology (e.g. \cite{AM97, AY13, G17, Y19GT, KNY}).   When $\alpha\cdot \alpha\geq 0$, adjunction inequalities were previously obtained in \cite{KM94, MST, OS2}. When $\alpha\cdot \alpha<0$, Ozsv\'{a}th and Szab\'{o} \cite{OS2} proved an adjunction inequality for four-manifolds satisfying a simple type condition on the extended Seiberg-Witten invariant, where in the $b_1=0$ case, their simple type condition is the same as ours. Applying Corollary \ref{main} and results of Ozsv\'{a}th and Szab\'{o} \cite{OS1, OS2}, we obtain adjunction inequalities without assuming any simple type condition.

\begin{theorem}
  \label{adjunction}
  Let $p$ be a prime, and suppose that $b_2^+$ is odd with $b_2^+>2$, $b_1=0$, $(b_2^+-1)/2\not\equiv 0\mod p$ and $\mathfrak{s}$ is a mod $p$ basic class. Let $\alpha$ be a second homology class of $X$ which satisfies $\alpha\cdot\alpha<0$ and is represented by a smoothly embedded closed oriented surface of genus $g$.
  \begin{enumerate}
    \item If $g\ge 2p-3$, then
    \[
      |\langle c_1(\mathfrak{s}),\alpha\rangle|+\alpha\cdot\alpha+2d(\mathfrak{s})\le 2g-2.
    \]
    \item If $g\ge p-1$, then
    \[
      |\langle c_1(\mathfrak{s}),\alpha\rangle|+\alpha\cdot\alpha+d(\mathfrak{s})\le 2g-2.
    \]
  \end{enumerate}
\end{theorem}

\begin{remark}
  It is straightforward to generalize Theorem \ref{adjunction} for mod $p^r$ basic classes by using Theorem \ref{main2}.
\end{remark}

We note that the $p=2$ case of this theorem is a special case of a result of Kato, Nakamura and Yasui \cite[Theorem 1.7]{KNY}. On the other hand, we can also derive adjunction inequalities for a basic class, instead of a mod $p$ basic class. For a basic class $\mathfrak{s}$, as seen from the proof of Corollary \ref{non-prime bound}, we can find a prime $p$ with $p\leq\max\{|SW(\mathfrak{s})|-1,\, (b_2^+-3)/2,\, 7\}$ satisfying the assumption of the above theorem. Hence, the adjunction inequalities in (1) and (2) hold for any ordinary basic class $\mathfrak{s}$ if $g\geq \max\{|2SW(\mathfrak{s})|-5,\, b_2^+-6,\, 11\}$ and $g\geq \max\{|SW(\mathfrak{s})|-2,\, (b_2^+-5)/2,\, 6\}$, respectively. As well as Corollary \ref{non-prime bound}, these conditions get better as either $|SW(\mathfrak{s})|$ or $b_2^+$ get larger (see Remark \ref{remark upper bound}).

\subsection*{Acknowledgements}

The authors were supported in part by JSPS KAKENHI Grant Numbers 17K05248 and 19K03473 (Kishimoto), 17H02841 and  22H01123 (Kato), 19K03506 (Nakamura), 19H01788 and 19K03491 (Yasui).

\section{Upper bound}\label{Upper bound}

This section proves Theorem \ref{main2} and Corollary \ref{non-prime bound}. Hereafter, let $X$ be a closed, connected, oriented and smooth four-manifold with $b_1=0$ and $b_2\ge 2$, and let $\mathfrak{s}$ be a $\spin^c$ structure on $X$.

We use the Bauer-Furuta invariant. Since there is a formula
\begin{equation}
  \label{d(s)}
  d(\mathfrak{s})=\frac{c_1(\mathfrak{s})^2-\mathrm{sign}(X)}{4}-(1+b_2^+),
\end{equation}
it follows from \cite[Proposition 3.4]{BF} that the Bauer-Furuta invariant $BF(\mathfrak{s})$ belongs to the stable cohomotopy group $\pi^{b_2^+-1}(\C P^{d-1})$, where $2d=d(\mathfrak{s})+1+b_2^+$. Moreover, there is an identity
\begin{equation}
  \label{BF-SW}
  \mathrm{hur}(BF(\mathfrak{s}))=SW(\mathfrak{s}),
\end{equation}
where $\mathrm{hur}\colon\pi^{b_2^+-1}(\C P^{d-1})\to H^{b_2^+-1}(\C P^{d-1})$ denotes the Hurewicz homomorphism. Note that $H^{b_2^+-1}(\C P^{d-1})$ is isomorphic to $0$ and $\Z$ according to $b_2^+$ being even and odd.

{Let $\Z_{(p)}$ denote the localization of $\Z$ at the prime $p$, that is, it is the subring of $\Q$ consisting of fractions whose denominators are not divisible by $p$. We will prove the following theorem in Section \ref{Cohomotopy computation 2}

\begin{theorem}
  \label{cohomotopy CP'}
  Under the assumption in Theorem \ref{main2}, the natural map
  \[
  ( \pi^{2k}(\C P^{k+r(p-1)})
    \otimes \Z_{(p)} )/ \mathrm{Tor}
    \to (\pi^{2k}(\C P^k) \otimes \Z_{(p)}) / \mathrm{Tor}
  \]
  is identified with $p^r\colon\Z_{(p)}\to\Z_{(p)}$.
\end{theorem}

\begin{proof}
  [Proof of Theorem \ref{main2}]
  Since $SW(\mathfrak{s})=0$ for $b_2^+$ even, we only need to consider the case $b_2^+$ odd. Let $2n=b_2^++2r(p-1)-1$ and $2\delta=d(\mathfrak{s})-2r(p-1)$. Then the Bauer-Furuta invariant $BF(\mathfrak{s})$ belongs to $\pi^{2n-2r(p-1)}(\C P^{n+\delta})$ as mentioned above. Suppose $d(\mathfrak{s})\ge 2r(p-1)$. Then $\delta\ge 0$, and so there is a commutative diagram
  \[
    \xymatrix{
      \pi^{2n-2r(p-1)}(\C P^{n+\delta})\ar[r]^{\mathrm{hur}}\ar[d]_{i_2^*}&H^{2n-2r(p-1)}(\C P^{n+\delta})\ar[d]^{i_2^*}_\cong\\
      \pi^{2n-2r(p-1)}(\C P^{n})\ar[r]^{\mathrm{hur}}\ar[d]_{i_1^*}&H^{2n-2r(p-1)}(\C P^{n})\ar[d]^{i_1^*}_\cong\\
      \pi^{2n-2r(p-1)}(\C P^{n-r(p-1)})\ar[r]^{\mathrm{hur}}&H^{2n-2r(p-1)}(\C P^{n-r(p-1)}),
    }
  \]
  where $i_1\colon\C P^{n-r(p-1)}\to\C P^n$ and $i_2\colon\C P^n\to\C P^{n+\delta}$ are inclusions. 
We apply Theorem \ref{cohomotopy CP'} above. Then, the map
  \[
    i_1^*\otimes 1\colon
    (\pi^{2n-2r(p-1)}(\C P^{n})\otimes\Z_{(p)})/ \text{Tor}
    \to
    (
    \pi^{2n-2r(p-1)}(\C P^{n-r(p-1)})\otimes\Z_{(p)}) / \text{Tor}
  \]
  is identified with $p^r\colon\Z_{(p)}\to\Z_{(p)}$, the multiplication by $p^r$, where $\Z_{(p)}$ denotes the ring of all rational numbers whose denominators are not divisible by $p^r$. Then the image of the Hurewicz homomorphism
  \[
  \mathrm{hur}\colon\pi^{2n-2r(p-1)}(\C P^{n+\delta})\to H^{2n-2r(p-1)}(\C P^{n+\delta})
  \]
   is included in $p^r\Z\subset\Z\cong H^{2n-2r(p-1)}(\C P^{n+\delta})$. Thus by the identity \eqref{BF-SW}, we obtain that $SW(\mathfrak{s})$ is divisible by $p^r$, which contradicts to the assumption that $\mathfrak{s}$ is a mod $p^r$ basic class. Thus we must have $d(\mathfrak{s})\le 2r(p-1)-1$. Since $\mathfrak{s}$ is a basic class, $d(\mathfrak{s})$ is even. Therefore we obtain $d(\mathfrak{s})\le 2r(p-1)-2$, proving the statement.
\end{proof}

We prove the crucial part of Corollary \ref{non-prime bound} in a more general form. To this end, we consider primes in intervals. For $c\in(1/2,1]$, let $S_{c,n}$ be the infinimum of real numbers 
such that for each $x\ge S_{c,n}$, there are at least $n$ primes in $(x/2,x)$ for $c=1$ and $(x/2,cx]$ for $c<1$. Lemma \ref{S-R} guarantees that $S_{c,n}$ certainly exists for each $c,n$, where we can easily see that $S_{c,n}$ is actually the least real number satisfying the above condition.

\begin{theorem}
  \label{upper bound S}
  Let $n=\max\{|SW(\mathfrak{s})|,(b_2^+-1)/2\}$ and $c\in(1/2,1]$. If $\mathfrak{s}$ is a basic class and $n\ge S_{c,2}$, then
  \[
    d(\mathfrak{s})\le
    \begin{cases}
      2n-6&(c=1)\\
      2cn-4&(c<1).
    \end{cases}
  \]
\end{theorem}

\begin{proof}
  Since $n\ge S_{c,2}$, there are two primes $p,q$ in $(n/2,n)$ for $c=1$ and $(n/2,cn]\subset(n/2,n)$ for $c<1$. Then $p<n<2p$ and $q<n<2q$, implying that $n$ is not divisible by $p$ and $q$. Moreover, for any $1\le m\le n$, we have $m<pq$, implying $m$ is not divisible by at least one of $p,q$. Then $SW(\mathfrak{s})\cdot(b_2^+-1)/2$ is not divisible by either $p$ or $q$. Thus by Corollary \ref{uniform bound 2}, $d(\mathfrak{s})\le 2\max\{p,q\}-4$. Clearly, $\max\{p,q\}\le n-1$ for $c=1$ and $\max\{p,q\}\le cn$ for $c<1$, completing the proof.
\end{proof}

\begin{remark}
  For $n<S_{c,2}$, we can alternatively apply Corollary \ref{non-prime bound} to get an upper bound for $d(\mathfrak{s})$.
\end{remark}

To make Theorem \ref{upper bound S} applicable, we give an upper bound for $S_{c,n}$ in terms of a generalized Ramanujan prime introduced in \cite{ABMRS}. For $c\in(0,1)$, the $n$-th $c$-Ramanujan prime $R_{c,n}$ is defined to be the least number such that for any $x\ge R_{c,n}$, $(cx,x]$ includes at least $n$ primes. Clearly, $R_{c,n}$ is a prime, and $R_{\frac{1}{2},n}$ coincides with the Ramanujan prime $R_n$ (see \cite{S}). It is proved in \cite[Theorem 2.2]{ABMRS} that $R_{c,n}$ exists for each $c,n$, and as in \cite[Section 1]{ABMRS}, for $n=1,2,3,4,5$, we have
\begin{equation}
  \label{R}
  R_{\frac{1}{2},n}=2,11,17,29,41\quad\text{and}\quad R_{\frac{3}{4},n}=11,29,59,67,101.
\end{equation}

\begin{lemma}
  \label{S-R}
  There are inequalities
  \[
    S_{1,n}\le R_{\frac{1}{2},n+1}\quad\text{and}\quad S_{c,n}\le\frac{1}{c}R_{\frac{1}{2c},n},
  \]
  where $c\in(1/2,1)$.
\end{lemma}

\begin{proof}
  If $x\ge R_{\frac{1}{2},n+1}$, then $(x/2,x]$ includes at least $n+1$ primes, implying that $(x/2,x)$ includes at least $n$ primes. Hence the first inequality is proved. Let $c\in(1/2,1)$. If $cx\ge R_{\frac{1}{2c},n}$, then there are at least $n$ primes in $(x/2,cx]$, implying the second inequality.
\end{proof}

We give a coarse upper bound for $R_{c,n}$, which helps evaluate $S_{c,n}$ by Lemma \ref{S-R}.

\begin{lemma}
  \label{R upper bound}
  For $n\ge 1$ and $c\in(0,1)$, there is an inequality
  \[
    R_{c,n}\le\max\left\{(2\lceil\sqrt{2n}+1\rceil)!,\,\mathrm{exp}\left(\frac{-\log c+\frac{3}{2}}{1-c}\right),\,\frac{e^\frac{3}{2}}{c},\,59\right\},
  \]
  where $\lceil x\rceil$ denotes the least integer $\ge x$.
\end{lemma}

\begin{proof}
  Let $\pi(x)$ denote the prime counting function, that is, $\pi(x)$ is the number of primes $\le x$. By \cite[Theorem 2 and Corollary 1]{RS}, for $x>\max\{59,e^\frac{3}{2}/c\}$, we have
  \[
    \pi(x)-\pi(cx)>\frac{x}{\log x}\left(1+\frac{1}{2\log x}\right)-\frac{cx}{\log(cx)-\frac{3}{2}}.
  \]
  If $x\ge\mathrm{exp}\left(\frac{-\log c+\frac{3}{2}}{1-c}\right)$, then $\frac{x}{\log x}-\frac{cx}{\log(cx)-\frac{3}{2}}\ge 0$, and so
  \[
    \pi(x)-\pi(cx)\ge\frac{x}{2(\log x)^2}.
  \]
  Clearly, $\frac{x}{2(\log x)^2}\ge n$ if and only if $e^{\sqrt{x}}-x^{\sqrt{2n}}\ge 0$. For $x\ge(2\lceil\sqrt{2n}+1\rceil)!$, we have
  \[
    e^{\sqrt{x}}-x^{\sqrt{2n}}>\frac{x^{\lceil\sqrt{2n}+1\rceil}}{(2\lceil\sqrt{2n}+1\rceil)!}-x^{\sqrt{2n}}\ge 0.
  \]
  Thus the proof is complete.
\end{proof}

Now we are ready to prove Corollary \ref{non-prime bound}.

\begin{proof}
  [Proof of Corollary \ref{non-prime bound}]
  Let $n=\max\{|SW(\mathfrak{s})|,(b_2^+-1)/2\}$. By \eqref{R} and Lemma \ref{S-R}, we get $S_{1,2}\le R_{1,3}=17$. Then by Theorem \ref{upper bound S}, we obtain the inequality in the statement for $n\ge 17$. We can easily check that if $12\le n\le 16$, then there are two primes in $(n/2,n)$, so that the proof of Theorem \ref{upper bound S} for $c=1$ works verbatim to show the inequality holds for $12\le n\le 16$. Suppose $1\le m\le n\le 11$. Then $m,n$ are divisible by at most two of $2,3,5,7,11$. Moreover, at most one of $m,n$ is divisible by $11$, and if this is the case, $mn$ is not divisible by at least one of $2,3,5,7$. We also have that at most one of $m,n$ is divisible by $7$, and if this is the case, $mn$ is not divisible by at least one of $2,3,5$. Then we obtain that $SW(\mathfrak{s})\cdot(b_2^+-1)/2$ is not divisible by at least one of $2,3,5,7$. Thus by Corollary \ref{uniform bound 2}, $d(\mathfrak{s})\le 2\cdot 7-4=10$, completing the proof.
\end{proof}

Note that Proposition \ref{upper bound S} implies that if we assume either $|SW(\mathfrak{s})|$ or $b_2^+$ is large enough, then we could get a sharper upper bound than Corollary \ref{non-prime bound}. Here, we give such an example.

\begin{example}
  \label{n large}
  Let $n=\max\{|SW(\mathfrak{s})|,(b_2^+-1)/2\}$, and let $\mathfrak{s}$ be a basic class. By \eqref{R} and Lemma \ref{S-R}, $S_{\frac{2}{3},2}\le\frac{3}{2}R_{\frac{3}{4},2}=\frac{3}{2}\cdot 29=43.5$, and so by Theorem \ref{upper bound S}, for $n\ge 44$, there is an inequality
  \[
    d(\mathfrak{s})\le\frac{4}{3}n-4=\max\left\{\frac{4}{3}|SW(\mathfrak{s})|-4,\,\frac{2}{3}b_2^+-\frac{14}{3}\right\}.
  \]
\end{example}

\section{Adjunction inequality}\label{Adjunction inequality}

This section proves Theorem \ref{adjunction}. To this end, we use the Seiberg-Witten invariant of the form
\[
  SW_\mathfrak{s}\colon\mathbb{A}(X)\to\Z,
\]
where $\mathbb{A}(X)=(\Lambda H_1(X;\Z))\otimes\Z[U]$ such that elements of $H_1(X;\Z)$ are assumed to be of degree 1 and $U$ is of degree 2. See \cite{OS1,OS2} for details. The above Seiberg-Witten invariant is an extension of the usual Seiberg-Witten invariant because there is an identity
\begin{equation*}
  \label{SW(U)}
  SW_\mathfrak{s}(U^{{d(\mathfrak{s})}/{2}})=SW(\mathfrak{s})
\end{equation*}
whenever $d(\mathfrak{s})$ is even. 

Let $\alpha$ be a second homology class of $X$ represented by a smoothly embedded closed oriented surface $\Sigma$ of genus $g$. Let $\mathrm{PD}(\alpha)$ denote the Poincar\'e dual of $\alpha$. Since $c_1(\mathfrak{s}+\mathrm{PD}(\alpha))=c_1(\mathfrak{s})+2\mathrm{PD}(\alpha)$, it follows from \eqref{d(s)} that
\begin{equation}
  \label{formula d}
  \begin{split}
    d(\mathfrak{s}+\mathrm{PD}(\alpha))&=d(\mathfrak{s})+c_1(\mathfrak{s})\mathrm{PD}(\alpha)+\mathrm{PD}(\alpha)^2\\
    &=d(\mathfrak{s})+\langle c_1(\mathfrak{s}),\alpha\rangle+\alpha\cdot\alpha.
  \end{split}
\end{equation}
Since $c_1(\mathfrak{s})\equiv w_2(M)\mod 2$, the Wu formula implies that $\langle c_1(\mathfrak{s}),\alpha\rangle+\alpha\cdot\alpha$ is an even integer. In particular, $d(\mathfrak{s}+\mathrm{PD}(\alpha))$ is even whenever so is $d(\mathfrak{s})$. We will freely use these facts.

\begin{lemma}
  \label{mod p basic form 1}
  Suppose that $-\langle c_1(\mathfrak{s}),\alpha\rangle+\alpha\cdot\alpha\ge\max\{2g-2d(\mathfrak{s}),0\}$ and $g\ge 1$. If $\mathfrak{s}$ is a mod $p$ basic class, then so is $\mathfrak{s}-\mathrm{PD}(\alpha)$ too.
\end{lemma}

\begin{proof}
  By \eqref{formula d} and $-\langle c_1(\mathfrak{s}),\alpha\rangle+\alpha\cdot\alpha\ge 0$, there is an inequality $d(\mathfrak{s}-\mathrm{PD}(\alpha))\ge d(\mathfrak{s})$. Then since $-\langle c_1(\mathfrak{s}),\alpha\rangle+\alpha\cdot\alpha+2d(\mathfrak{s})\ge 2g$ and $g\ge 1$, we can apply \cite[Theorem 1.7]{OS2} to get
    \[
    SW_{\mathfrak{s}-\mathrm{PD}(\alpha)}(U^{d(\mathfrak{s}-\mathrm{PD}(\alpha))/2})=SW_\mathfrak{s}(U^{d(\mathfrak{s})/2}).
  \]
  (In \cite[Theorem 1.7]{OS2}, it is stated that $SW_\mathfrak{s}(U^{d})=SW_{\mathfrak{s}-\mathrm{PD}[\Sigma]}(U^{d^\prime})$, where $d$ and $d^\prime$ denote the dimensions of $\mathfrak{s}$ and $\mathfrak{s}-\mathrm{PD}[\Sigma]$, respectively. We must be aware this is a typo by dimensionality, and the above equality is the correct one.) Thus by \eqref{SW(U)}, $SW(\mathfrak{s}-\mathrm{PD}(\alpha))=SW(\mathfrak{s})$, implying that $SW(\mathfrak{s}-\mathrm{PD}(\alpha))$ is a mod $p$ basic class, as stated.
\end{proof}

When $d(\mathfrak{s})=2n\ge 0$, e.g. $\mathfrak{s}$ is a basic class, we define
\[
  \widehat{X}=X\# n\overline{\C P^2}.
\]
Clearly, $b_1(\widehat{X})=b_1(X)=0$ and $b_2^+(\widehat{X})=b_2^+(X)\ge 2$. We may regard that the (co)homology of $X$ is a subgroup of the (co)homology of $\widehat{X}$. Let $L=c_1(\mathfrak{s})+3\mathrm{PD}(e_1)+\cdots+3\mathrm{PD}(e_n)\in H^2(\widehat{X})$, where each $e_i$ is the second homology class of the $i$-th $\overline{\C P^2}$ represented by the exceptional sphere. By the blow-up formula \cite{FS,N}, there is a $\spin^c$ structure $\hat{\mathfrak{s}}$ on $\widehat{X}$ such that
\[
  c_1(\hat{\mathfrak{s}})=L\quad\text{and}\quad SW(\hat{\mathfrak{s}})=SW(\mathfrak{s}).
\]
Since $\text{sign}(\widehat{X})=\text{sign}(X)-n$, it follows from \eqref{d(s)} that
\begin{align*}
  d(\hat{\mathfrak{s}})&=\frac{c_1(\hat{\mathfrak{s}})^2-\text{sign}(\widehat{X})}{4}-(1+b_2^+(\widehat{X}))\\
  &=\frac{c_1(\mathfrak{s})^2-\text{sign}(X)}{4}-(1+b_2^+(X))-2n\\
  &=d(\mathfrak{s})-2n\\
  &=0.
\end{align*}
Let $\hat{\alpha}=\alpha-e_1-\cdots-e_n\in H_2(\widehat{X})$.

\begin{lemma}
  \label{mod p basic form 2}
  Suppose that $\alpha\cdot\alpha<0$, $\langle c_1(\mathfrak{s}),\alpha\rangle+\alpha\cdot\alpha\ge 2g$ and $g\ge 1$. If $\mathfrak{s}$ is a mod $p$ basic class, then so is $\hat{\mathfrak{s}}-\mathrm{PD}(\hat{\alpha})$ too.
\end{lemma}

\begin{proof}
  Since $\mathfrak{s}$ is a basic class, we can consider $\widehat{X}$ and $\hat{\mathfrak{s}}$. We may assume that $\Sigma$ is also embedded into $\widehat{X}$. Let
  \[
    \xi(\Sigma)=(U-x_1y_1)\cdots(U-x_gy_g)\in\mathbb{A}(\widehat{X}),
  \]
  where $\{x_1,y_1,\ldots,x_g,y_g\}$ is the image of the standard symplectic basis of $H_1(\Sigma)$. Then by assumption and $d(\hat{\mathfrak{s}})=0$, we can apply \cite[Theorem 1.3]{OS1} to obtain
  \[
    SW_{\hat{\mathfrak{s}}+\mathrm{PD}(\hat{\alpha})}(\xi(\Sigma)U^m)=SW_{\hat{\mathfrak{s}}}(1),
  \]
  where $2m=\langle c_1(\mathfrak{s}),\alpha\rangle+\alpha\cdot\alpha-2g\ge 0$. Since $b_1(\widehat{X})=0$, $i_*(\xi(\Sigma))U^m$ coincides with $U^{g+m}$ modulo torsion elements. Since $SW_{\hat{\mathfrak{s}}+\mathrm{PD}(\hat{\alpha})}\colon\mathbb{A}(\widehat{X})\to\Z$ is linear, it annihilates torsion elements, so that $SW_{\hat{\mathfrak{s}}+\mathrm{PD}(\hat{\alpha})}(\xi(\Sigma)U^m)=SW_{\hat{\mathfrak{s}}+\mathrm{PD}(\hat{\alpha})}(U^{g+m})$. Thus by \eqref{SW(U)},
  \[
    SW(\hat{\mathfrak{s}}+\mathrm{PD}(\hat{\alpha}))=SW_{\hat{\mathfrak{s}}+\mathrm{PD}(\hat{\alpha})}(U^{m+g})=SW_{\hat{\mathfrak{s}}}(1)=SW(\hat{\mathfrak{s}})=SW(\mathfrak{s}),
  \]
  implying $\hat{\mathfrak{s}}+\mathrm{PD}(\hat{\alpha})$ is a mod $p$ basic class, as desired.
\end{proof}

We are ready to prove Theorem \ref{adjunction}.

\begin{proof}
  [Proof of Theorem \ref{adjunction}]
  (1) By Corollary~\ref{main} and the assumption, we have $2d(\mathfrak{s})\leq 2(2p-4)\leq 2g-2$. It thus suffices to prove the case $|\langle c_1(\mathfrak{s}),\alpha\rangle|+\alpha\cdot\alpha\ge 0$. By reversing the orientation of the embedded surface if necessary, we may assume $\langle c_1(\mathfrak{s}),\alpha\rangle\le 0$, so that $-\langle c_1(\mathfrak{s}),\alpha\rangle+\alpha\cdot\alpha\ge 0$. We also have $g\ge 2p-3\ge 1$. Assume that $-\langle c_1(\mathfrak{s}),\alpha\rangle+\alpha\cdot\alpha+2d(\mathfrak{s})\ge 2g-1$. Then since $-\langle c_1(\mathfrak{s}),\alpha\rangle+\alpha\cdot\alpha+2d(\mathfrak{s})$ is even, $-\langle c_1(\mathfrak{s}),\alpha\rangle+\alpha\cdot\alpha+2d(\mathfrak{s})\ge 2g$. So by Lemma \ref{mod p basic form 1}, $\mathfrak{s}-\mathrm{PD}(\alpha)$ is a mod $p$ basic class. Moreover,
  \begin{align*}
    2d(\mathfrak{s}-\mathrm{PD}(\alpha))&=2(-\langle c_1(\mathfrak{s}),\alpha\rangle+\alpha\cdot\alpha+d(\mathfrak{s}))\\
    &\ge -\langle c_1(\mathfrak{s}),\alpha\rangle+\alpha\cdot\alpha+2d(\mathfrak{s})\\
    &\ge 2g\\
    &\ge 2(2p-3).
  \end{align*}
  Then we obtain a contradiction to Corollary \ref{main}. Therefore we must have $-\langle c_1(\mathfrak{s}),\alpha\rangle+\alpha\cdot\alpha+2d(\mathfrak{s})\le 2g-2$.

  \noindent(2) Assume that $\langle c_1(\mathfrak{s}),\alpha\rangle+\alpha\cdot\alpha+d(\mathfrak{s})\ge 2g$. Then by Lemma \ref{mod p basic form 2}, $\hat{\mathfrak{s}}+\mathrm{PD}(\hat{\alpha})$ is a mod $p$ basic class. Moreover, by \eqref{formula d},
  \begin{align*}
    d(\hat{\mathfrak{s}}+\mathrm{PD}(\hat{\alpha}))&=\langle c_1(\hat{\mathfrak{s}}),\hat{\alpha}\rangle+\hat{\alpha}\cdot\hat{\alpha}+d(\hat{\mathfrak{s}})\\
    &=(\langle c_1(\mathfrak{s}),\alpha\rangle+3n)+(\alpha\cdot\alpha-n)+0\\
    &=\langle c_1(\mathfrak{s}),\alpha\rangle+\alpha\cdot\alpha+d(\mathfrak{s})\\
    &\ge 2g\\
    &\ge 2p-2,
  \end{align*}
  where $d(\mathfrak{s})=2n$. So we obtain a contradiction to Corollary \ref{main}, and hence $\langle c_1(\mathfrak{s}),\alpha\rangle+\alpha\cdot\alpha+d(\mathfrak{s})\le 2g-1$. Since $\langle c_1(\mathfrak{s}),\alpha\rangle+\alpha\cdot\alpha+d(\mathfrak{s})$ is even, we must have $\langle c_1(\mathfrak{s}),\alpha\rangle+\alpha\cdot\alpha+d(\mathfrak{s})\le 2g-2$. By reversing the orientation of the embedded surface if necessary, we may assume $\langle c_1(\mathfrak{s}),\alpha\rangle\ge 0$. Thus the proof is complete.
\end{proof}

\section{Cohomotopy computation}
\label{Cohomotopy computation 2}

This section proves Theorem \ref{cohomotopy CP'}. Since we are concerned with stable cohomotopy groups, all spaces and maps will be stabilized. We will also localize all spaces and maps at an odd prime $p$, unless otherwise is specified. 
We refer to \cite{HMR} for $p$-localization.

\subsection{Reduction}

We reduce the computation of this map to simpler spaces. The following ($p$-locally stable) splitting was proved in \cite[Theorem 9.3 and Corollary 9.5]{MNT}.

\begin{lemma}
  \label{decomposition}
  There is a homotopy equivalence
  \[
    \C P^n\simeq X_1^r\vee\cdots\vee X_{p-1}^r
  \]
  such that $X_i^r=S^{2i}\cup e^{2i+2(p-1)}\cup\cdots\cup e^{2i+2r(p-1)}$, where $r=\left[\frac{n-i}{p-1}\right]$.
\end{lemma}

%

We can easily see from the proof of \cite[Theorem 9.3]{MNT} that the splitting of Lemma \ref{decomposition} is natural with respect to $n$ in the sense that $X_i^r$ for $\C P^n$ is a subcomplex of $X_i^r$ for $\C P^{n+1}$, where $i=1,\ldots,p-1$. Then for $s\le r$, we can consider the quotient
\[
  X_i^r/X_i^s=S^{2i+2(s+1)(p-1)}\cup e^{2i+2(s+2)(p-1)}\cup\cdots\cup e^{2i+2r(p-1)}.
\]

\begin{lemma}
  \label{torsion}
  If $i\not\equiv k\mod p-1$, then $\pi^{2k}(X_i^r)$ is a finite abelian group.
\end{lemma}

\begin{proof}
  The cofiber sequence $X^{r-1}_i\to X^r_i \to X^r_i/X^{r-1}_i=S^{2i+2r(p-1)}$ induces a long exact sequence of cohomotopy groups
  \[
    \cdots\to\pi^{2k}(S^{2i+2r(p-1)})\to \pi^{2k}(X^r_i)\to \pi^{2k}(X^{r-1}_i)\to\cdots.
  \]
  Then since $\pi^*(S^0)$ is a finite abelian group for $*>0$, the statement is proved by induction on $r$.
\end{proof}

Hereafter, we set $k=i+s(p-1)$ for given integers $1\le i\le p-1$ and $s\ge 0$. Let $Y_k^t=X_i^{s+t}/X_i^{s-1}$. Then we have
\[
  Y_k^t=S^{2k}\cup e^{2k+2(p-1)}\cup\cdots\cup e^{2k+2t(p-1)}.
\]

\begin{lemma}
  \label{X-Y}
The natural map $\pi^{2k}(Y_k^t)\to\pi^{2k}(X_i^{s+t})$ is an isomorphism. Moreover, there is an isomorphism
\[
  \pi^{2k}(Y_k^t)/\mathrm{Tor}\cong \Z_{(p)}.
\]
\end{lemma}

\begin{proof}
Since the dimension of $X_i^{s-1}$ is smaller than $2k$, we have $\pi^{2k}(X_i^{s-1})=0$. Then by the long exact sequence
\[
  \cdots\to\pi^*(Y^t_k)\to\pi^*(X_i^{s+t})\to\pi^*(X_i^{s-1})\to\cdots
\]
we obtain the first isomorphism. Consider the long exact sequence
\[
  \cdots\to\pi^*(Y_{k+(p-1)}^t)\to\pi^*(Y_k^t)\to\pi^*(S^{2k})\to\cdots.
\]
Then since $Y_{k+p-1}^t$ is $(2(k+p-1)-1)$-connected and $\pi^{2k}(S^{2k})\cong\Z_{(p)}$,  we obtain the second isomorphism.
\end{proof}

Now we are ready to prove:

\begin{proposition}
  \label{reduction}
  There is a commutative diagram
  \[
    \xymatrix{
      \pi^{2k}(\C P^n)/\mathrm{Tor}\ar[r]\ar[d]_\cong&\pi^{2k}(\C P^{n-p+1})/\mathrm{Tor}\ar[d]^\cong\\
      \pi^{2k}(Y_k^r)/\mathrm{Tor}\ar[r]&\pi^{2k}(Y_k^{r-1})/\mathrm{Tor}
    }
  \]
  where the horizontal maps are induced from inclusions and $r=\left[\frac{n-i}{p-1}\right]$.
\end{proposition}

\begin{proof}
  By Lemma \ref{decomposition}, there is a commutative diagram
  \[
    \xymatrix{
      \pi^{2k}(\C P^n)\ar[r]\ar[d]_\cong&\pi^{2k}(\C P^{n-p+1})\ar[d]^\cong\\
      \bigoplus_{i=1}^{p-1}\pi^{2k}(X_i^r)\ar[r]&\bigoplus_{i=1}^{p-1}\pi^{2k}(X_i^{r-1})
    }
  \]
  where the bottom map is the sum of the maps induced from the inclusions $X_i^{r-1}\to X_i^r$. Then the statement follows from Lemmas \ref{torsion} and \ref{X-Y}.
\end{proof}

By Proposition \ref{reduction}, we need to compute the map
\[
  j_r^*\colon\pi^{2k}(Y_k^r)/\mathrm{Tor}\to\pi^{2k}(Y_k^{r-1})/\mathrm{Tor}
\]
where $j_r\colon Y_k^{r-1}\to Y_k^r$ is the inclusion. To this end, we describe the attaching maps of cells of $Y_k^t$. To this end, we recall the ($p$-local stable) homotopy groups of $S^0$ in a range. Let $C\{x\}$ denotes an abelian group having a generator $x$ which is isomorphic with a cyclic group $C$.

\begin{theorem}[{\cite[Theorem 4.14]{T1}}]
  \label{homotopy group 2}
  For $i<p(p-1)$, we have
  \[
    \pi_{2i(p-1)-1}(S^0)\cong
    \begin{cases}
      \Z/p\{\alpha_i\}&i\not\equiv 0\mod p\\
      \Z/p^2\{\alpha_i'\}&i\equiv 0\mod p.
    \end{cases}
  \]
\end{theorem}

The following two lemmas describe the cell structure of $Y_k^r$.

\begin{lemma}
  \label{rank 1}   
If $k\equiv i-s\not\equiv 0\mod p$, then there is a homotopy equivalence
  \[
Y_k^1 =   X_i^{s+1}/X_i^{s-1}\simeq S^{2k}\cup_{\alpha_1}e^{2k+2(p-1)}.
  \]
\end{lemma}

\begin{proof}
By definition, there is a homotopy equivalence
  \[
    Y^1_k=X_i^{s+1}/X_i^{s-1}\simeq S^{2k}\cup_\phi e^{2k+2(p-1)}
  \]
  If the attaching map $\phi$ is non-trivial, then by Theorem \ref{homotopy group 2}, we can take $\phi=\alpha_1$. If $\phi$ is trivial, then the Steenrod operation $\mathcal{P}^1$ acts trivially on the mod $p$ cohomology of $Y^1_k=X_i^{s+1}/X_i^{s-1}$ because it is a wedge of spheres. Then it is sufficient to show that $\mathcal{P}^1$ acts non-trivially on the mod $p$ cohomology of $Y^1_k=X_i^{s+1}/X_i^{s-1}$. 
  Let $X=\C P^{k+(p-1)}/\C P^{k-(p-1)}$. Since $k\not\equiv 0\mod p$, we have
  \[
    \mathcal{P}^1(H^{2k}(X;\Z/p))=H^{2k+2(p-1)}(X;\Z/p).
  \]
  By Lemma \ref{decomposition}, the inclusion $Y^2_k=X_i^{s+1}/X_i^{s-1}\to X$ induces an isomorphism in the mod $p$ cohomology of dimension $2k,\,2k+2(p-1)$, implying that $\mathcal{P}^1$ acts non-trivially on the mod $p$ cohomology of $Y^2_k=X_i^{s+1}/X_i^{s-1}$. Thus the statement is proved.
\end{proof}

\begin{lemma}
  \label{rank 4}
  If $k=i+s(p-1)=ap+1$, then $Y^3_k=X_i^{s+3}/X_i^{s-1}$ is homotopy equivalent to
  \[
    S^{2k}\cup_{-\alpha_1}e^{{2k+2(p-1)}}\cup_{\left(\frac{a}{2}+1\right)\alpha_2}e^{{2k+4(p-1)}}\cup_{-\frac{a+1}{2}\alpha_2+\alpha_1}e^{{2k+6(p-1)}}.
  \]
\end{lemma}

\begin{proof}
  The lemma follows from \cite[Proposition 2.3]{D} and its proof.
\end{proof}

\subsection{Toda bracket}

Since our basic computation tool is the Toda bracket, we briefly recall its definition, where we refer to \cite{T1} for details. Suppose we are given maps
\[
  \gamma\colon W\to X,\quad\beta\colon X\to Y,\quad\alpha\colon Y\to Z
\]
satisfying $\beta\circ\gamma=0$ and $\alpha\circ\beta=0$. Let $h\colon CX\to Z$ be a null-homotopy for $\alpha\circ\beta$. Then we get a map
\[
  \bar{\alpha}=\alpha\cup h\colon Y\cup_\beta CX\to Z
\]
which is called an extension of $\alpha$ by $\beta$. Clearly, an extension of $\alpha$ by $\beta$ depends on the choice of a null-homotopy for $\alpha\circ\beta=0$. Let $\mathrm{Ext}(\alpha,\beta)$ denote the set of all extensions of $\alpha$ by $\beta$. Define a map $\tilde{\gamma}\colon\Sigma W\to Y\cup_\beta CX$ by
\[
  \tilde{\gamma}(y,t)=
  \begin{cases}
    g(y,1-2t)&0\le t\le\frac{1}{2}\\
    (\gamma(y),2t-1)&\frac{1}{2}\le t\le 1
  \end{cases}
\]
where $g\colon CW\to Y$ is a null-homotopy for $\beta\circ\gamma=0$. We call a map $\tilde{\gamma}$ a coextension of $\gamma$ by $\beta$. A coextension of $\gamma$ by $\beta$ depends on the choice of a null-homotopy for $\gamma\circ\beta=0$ as well as an extension above, and $\mathrm{Coext}(\beta,\gamma)$ denote the set of all coextensions of $\gamma$ by $\beta$.

\begin{definition}
  The Toda bracket of the above $\alpha,\beta,\gamma$ is defined as the set
  \[
    \mathrm{Ext}(\alpha,\beta)\circ\mathrm{Coext}(\beta,\gamma)\subset[\Sigma W,Z]
  \]
  which we denote by $\langle\alpha,\beta,\gamma\rangle$.
\end{definition}

We write $\langle\alpha,\beta,\gamma\rangle=\delta$ if the Toda bracket $\langle\alpha,\beta,\gamma\rangle$ consists of a single element $\delta$. As in \cite[Lemma 1.1]{T2}, the Toda bracket $\langle\alpha,\beta,\gamma\rangle$ is a coset of the subgroup
\[
  (\Sigma\gamma)^*([\Sigma X,Z])+\alpha_*([\Sigma W,Y])\subset[\Sigma W,Z]
\]
which is called the indeterminacy of $\langle\alpha,\beta,\gamma\rangle$. We prove a key lemma in our computation.

\begin{lemma}
  \label{Toda bracket}
  Suppose that maps $\beta\colon X\to Y$ and $\alpha\colon Y\to Z$ satisfy $k(\alpha\circ\beta)=0$ for some integer $k$. Then for any map $\gamma\colon\Sigma W\to Y\cup_\beta CX$ and any extension $\widetilde{k\alpha}\colon Y\cup_\beta CX\to Z$ of $k\alpha$ by $\beta$, we have
  \[
    \widetilde{k\alpha}\circ\gamma\in\langle k,\alpha\circ\beta,\Sigma^{-1}(\rho\circ\gamma)\rangle
  \]
  where $\rho\colon Y\cup_\beta CX\to\Sigma X$ denotes the pinch map.
\end{lemma}

\begin{proof}
  Since there is a homotopy cofibration $Y\cup_\beta CX\xrightarrow{\rho}\Sigma X\xrightarrow{\Sigma\beta}\Sigma Y$, we have $\beta\circ\Sigma^{-1}\rho=0$, implying $\beta\circ\Sigma^{-1}(\rho\circ\gamma)=0$. 
 Since we are stabilizing, $\gamma$ is a coextension of $\Sigma^{-1}(\rho\circ\gamma)$ by \cite{O}.
  Then since $k(\alpha\circ\beta)=0$, the Toda bracket $\langle k\alpha,\beta,\Sigma^{-1}(\rho\circ\gamma)\rangle$ is defined, and by definition, $\widetilde{k\alpha}\circ\gamma$ belongs to this Toda bracket. On the other hand, by \cite[Proposition 1.2]{T2}, we have
  \[
    \langle k\alpha,\beta,\Sigma^{-1}(\rho\circ\gamma)\rangle\subset\langle k,\alpha\circ\beta,\Sigma^{-1}(\rho\circ\gamma)\rangle.
  \]
  Thus the statement is proved.
\end{proof}

\subsection{Computation}

As in \cite{T1}, if we choose $\alpha_1\in\pi_{2p-3}(S^0)$, then $\alpha_i\in\pi_{2i(p-1)-1}(S^0)$ for $i>1$ are inductively defined by
\[
  \alpha_i=\langle\alpha_{i-1},p,\alpha_1\rangle.
\]
The element $\alpha_1$ is defined as a generator of $\pi_{2p-3}(S^0)=\Z/p$ with mod $p$ Hopf invariant $1$ (\cite[p.~309]{T1}).
If $i\equiv 0$ mod $p$, then $\pi_{2i(p-1)-1}(S^0)=\Z/{p^2}$ and $\alpha^\prime_i=\alpha_i/p$ is a generator.

We will use the following alternative description of $\alpha_i$

\begin{proposition}[{\cite[Proposition 4.17]{T1}}]
  \label{relation}
  Let $\alpha_{jp}=p\alpha'_{jp}$. If $s+t<p(p-1)$, then
  \[
    \langle p,\alpha_s,\alpha_t\rangle=
    \begin{cases}
      \frac{t}{s+t}\alpha_{s+t}&s+t\not\equiv 0\mod p\\
      \frac{pt}{s+t}\alpha_{s+t}'&s+t\equiv 0\mod p,
    \end{cases}
  \]
\end{proposition}

\begin{remark}
Since $\alpha_t\circ\alpha_s=0$ if $s+t<p(p-1)$ (\cite[Proposition 4.17]{T1}), the Toda bracket $\langle p,\alpha_s,\alpha_t\rangle$ in Proposition \ref{relation} is well-defined.
The statement of the theorem means that $\langle p,\alpha_s,\alpha_t\rangle$ has only one element given in the right hand side.
\end{remark}

Hereafter, we assume $t<p(p-1)$. Let $\varphi_t\colon S^{2k+2t(p-1)-1}\to Y_k^{t-1}$ denote the attaching map of the top cell of $Y_k^t=Y_k^{t-1}\cup e^{2k+2t(p-1)}$. We say that a map $\theta_{t-1}\colon Y_k^{t-1}\to S^{2k}$ detects $\varphi_t$, if the restriction of $\theta_{t-1}$ to the bottom cell $S^{2k}\subset Y_k^{t-1}$ is non-trivial and $\theta_{t-1}\circ\varphi_t$ generates $\pi^{2k}(S^{2k+2t(p-1)-1})=\pi_{2k+2t(p-1)-1}(S^{2k})$, where $\pi_{2k+2t(p-1)-1}(S^{2k})$ is given by Theorem \ref{homotopy group 2}. We define
\[
  q_t=
  \begin{cases}
    p&t\not\equiv 0\mod p\\
    p^2&t\equiv 0\mod p.
  \end{cases}
\]

\begin{lemma}
  \label{q_t}
  If a map $\theta_{t-1}\colon Y_k^{t-1}\to S^{2k}$ detects $\varphi_t$, then the map $j_t^*\colon\pi^{2k}(Y_k^t)/\mathrm{Tor}\to\pi^{2k}(Y_k^{t-1})/\mathrm{Tor}$ is identified with the map
  \[
    q_t\colon\Z_{(p)}\to\Z_{(p)}.
  \]
\end{lemma}

\begin{proof}
  Consider the exact sequence
  \[
    \cdots\to\pi^{2k}(Y_k^t)\to\pi^{2k}(Y_k^{t-1})\xrightarrow{\varphi_t^*}\pi^{2k}(S^{2k+2t(p-1)-1})\to\cdots
  \]
  induced from the cofibration sequence $S^{2k+2t(p-1)-1}\xrightarrow{\varphi_t}Y_k^{t-1}\to Y_k^t$. By Lemma \ref{X-Y}, we have $\pi^{2k}(Y_k^t)/\mathrm{Tor}\cong\pi^{2k}(Y_k^{t-1})/\mathrm{Tor}\cong\Z_{(p)}$, and by Theorem \ref{homotopy group 2}, we also have $\pi^{2k}(S^{2k+2t(p-1)-1})\cong\Z/q_t$. Then it is sufficient to show that there is an element $\phi\in\pi^{2k}(Y_k^{t-1})$ of infinite order such that $\varphi_t^*(\phi)$ generates $\pi^{2k}(S^{2k+2t(p-1)-1})$. Since $\theta_{t-1}\vert_{S^{2k}}\ne 0$ and $\pi^{2k}(S^{2k})\cong\Z_{(p)}$, $\theta_{t-1}$ is of infinite order. 
  Moreover, $\varphi_t^*(\theta_{t-1})$ generates $\pi^{2k}(S^{2k+2t(p-1)-1})$, because $\theta_{t-1}$ detects $\varphi_t$. This 
  completes the proof.
\end{proof}

\begin{lemma}
  \label{composition rank 1}
  If $k-t+1\not\equiv 0\mod p$ with $t\ge 2$ and there is a map $\theta_{t-2}\colon Y_k^{t-2}\to S^{2k}$ detecting $\varphi_{t-1}$, then there is a map $\theta_{t-1}\colon Y_k^{t-1}\to S^{2k}$ detecting $\varphi_{t}$.
\end{lemma}

\begin{proof}
  By Theorem \ref{homotopy group 2}, we may assume
  \[
    \theta_{t-2}\circ\varphi_{t-1}=
    \begin{cases}
      \alpha_{t-1}&t-1\not\equiv 0\mod p\\
      \alpha'_{t-1}&t-1\equiv 0\mod p.
    \end{cases}
  \]
  We define $\phi=\frac{q_{t-1}}{p}\circ\theta_{t-2}$. Then we have $p\circ\phi\circ\varphi_{t-1}=q_{t-1}\circ\theta_{t-2}\circ\varphi_{t-1}=q_{t-1}(\theta_{t-2}\circ\varphi_{t-1})=0$ 
  because we are stabilizing. 
  Hence we can set $\theta_{t-1}\colon Y_k^{t-1}=Y_k^{t-2}\cup_{\varphi_{t-1}}e^{2k+2(t-1)(p-1)} \to S^{2k}$ 
  to be an extension of $p\circ\phi: Y^{t-2}_k \to S^{2k}$ by $\varphi_{t-1}$. 
We apply  Lemma \ref{Toda bracket} for $\alpha=\phi$, $\beta=\varphi_{t-1}$ and $\gamma=\varphi_t$.
Then   the composite $\theta_{t-1}\circ\varphi_t$ belongs to 
   the Toda bracket $\langle p,\phi\circ\varphi_{t-1},\Sigma^{-1}(\rho\circ\varphi_t)\rangle$,
   where $\rho\colon Y^{t-1}_k\to Y^{t-1}_k/Y^{t-2}_k = S^{2k+2(t-1)(p-1)}$ is the pinch map onto the top cell.
Note that $Y^{t}_k/Y^{t-2}_k= Y^1_{k+(t-1)(p-1)} = S^{2k+2(t-1)(p-1)}\cup_{\rho\circ\varphi_t} e^{2k+2t(p-1)}$.
Since $k+(t-1)(p-1)\not\equiv 0$ mod $p$ by the assumption, Lemma \ref{rank 1} implies that $\rho\circ\varphi_t=\alpha_1$.
We may assume $\Sigma^{-1}(\rho\circ\varphi_t) = \alpha_1$ because we are stabilizing.
   We also have
  \[
    \phi\circ\varphi_{t-1}=\frac{q_{t-1}}{p}\circ\theta_{t-2}\circ\varphi_{t-1}=\frac{q_{t-1}}{p}(\theta_{t-2}\circ\varphi_{t-1})=\alpha_{t-1}
  \]
  because we are stabilizing, where $p\alpha'_{t-1}=\alpha_{t-1}$ for $t-1\equiv 0\mod p$. 
  Then the composite $\theta_{t-1}\circ\varphi_t$ belongs to the Toda bracket 
  $\langle p,\alpha_{t-1},\alpha_1\rangle$ which is a subset of $\pi^{2k}(S^{2k+2t(p-1)-1})$, 
  where we have $q_t\circ\theta_{t-2}=p\circ(p\circ\theta_{t-2})$ 
  and $(p\circ\theta_{t-2})\circ\varphi_{t-1}=p\circ\alpha_t'=p\alpha'_t=\alpha_t$ 
  for $t-1\equiv 0\mod p$. On the other hand, it follows from Proposition \ref{relation} that
  \[
    \langle p,\alpha_{t-1},\alpha_1\rangle=
    \begin{cases}
      \frac{1}{t}\alpha_t&t\not\equiv 0\mod p\\
      \frac{p}{t}\alpha'_t&t\equiv 0\mod p
    \end{cases}
  \]
  where $\frac{p}{t}$ for $t\equiv 0\mod p$ makes sense because we are assuming $t<p(p-1)$. Then $\theta_{t-1}\circ\varphi_t$ generates $\pi^{2k}(S^{2k+2t(p-1)-1})$. By definition, we have 
  $\theta_{t-1}\vert_{S^{2k}}=q_t(\theta_{t-2}\vert_{S^{2k}})$. Since $\theta_{t-2}\vert_{S^{2k}}\ne 0$ 
  and $\pi^{2k}(S^{2k})\cong\Z_{(p)}$, $\theta_{t-2}\vert_{S^{2k}}$ is of infinite order, implying 
  $\theta_{t-1}\vert_{S^{2k}}\ne 0$. Thus the proof is finished.
\end{proof}

\begin{lemma}
  \label{detection}
  Let $t\ge 4$ and $k-t+2\equiv 0\mod p$. If there is a map $\theta_{t-4}\colon Y_k^{t-4}\to S^{2k}$ detecting $\varphi_{t-3}$, then there is a map $\theta_{t-1}\colon Y_k^{t-1}\to S^{2k}$ such that $\theta_{t-1}\vert_{S^{2k}}$ is non-trivial and
  \[
    \theta_{t-1}\circ\varphi_t=
    \begin{cases}
      -\frac{p(3a(k,t)+5)}{2t}\alpha_t'&t\equiv 0\mod p\\
      \frac{p(a(k,t)+2)}{t-1}\alpha_t&t\equiv 1\mod p\\
      -\frac{p(a(k,t)+2)}{4(t-2)}\alpha_t&t\equiv 2\mod p\\
      \frac{3a(k,t)+4}{6}\alpha_t&t\equiv 3\mod p\text{ and }p>3\\
      \frac{1}{t}\left(\frac{a(k,t)+1}{t-2}+\frac{a(k,t)+2}{t-1}\right)\alpha_t&t\equiv 4,5,\ldots,p-1\mod p
    \end{cases}
  \]
  where $a(k,t)$ is as in \eqref{a(k,t)}.
\end{lemma}

\begin{proof}
  By Theorem \ref{homotopy group 2}, we may assume
  \begin{equation}\label{eq:t-4}
    \theta_{t-4}\circ\varphi_{t-3}=
    \begin{cases}
      \alpha'_{t-3}&t\equiv 3\mod p\\
      \alpha_{t-3}&t\not\equiv 3\mod p.
    \end{cases}
  \end{equation}
  Let us abbreviate $a(k,t)$ by $a$. 
  We can apply Lemma~\ref{rank 1} and Lemma \ref{rank 4} because $k+(t-3)(p-1)=ap+1$ and therefore $k+(t-4)(p-2)\not\equiv 0$ mod $p$. Then 
   there is a homotopy equivalence
  \begin{multline}
    \label{Y_k}
    Y_k^t\simeq Y_k^{t-4}\cup_{\alpha_1}e^{2k+2(t-3)(p-1)}\cup_{-\alpha_1}e^{2k+2(t-2)(p-1)}\\
    \cup_{\left(\frac{a}{2}+1\right)\alpha_2}e^{2k+2(t-1)(p-1)}\cup_{-\frac{a+1}{2}\alpha_2+\alpha_1}e^{2k+2t(p-1)}.
  \end{multline}
  Let $\theta_{t-3}\colon Y_k^{t-3}\to S^{2k}$ be an extension of $p\circ(\frac{q_t}{p}\theta_{t-4})$ by $\varphi_{t-4}$.

  \noindent(1) The $t\equiv 0\mod p$ case.

  As in the proof of Lemma \ref{composition rank 1}, we can see that $\theta_{t-3}\vert_{S^{2k}}\ne 0$.
  We apply Lemma \ref{Toda bracket} to the following setting
  \begin{gather*}
 \alpha=\frac{q_t}{p}\theta_{t-4}\colon Y^{t-4}_k\to S^{2k},\\
 \beta=\varphi_{t-3}=\alpha_1\colon S^{2k+2(t-3)(p-1)-1}\to Y^{t-4}_k,\\ 
\begin{aligned} \gamma=\varphi_{t-2}+\varphi_{t-1}=-\alpha_1\vee \left(\frac{a}{2}+1\right)\alpha_2\colon S^{2k+2(t-2)(p-1)-1}\vee S^{2k+2(t-1)(p-1)-1}\\ \to Y^{k-4}\cup_{\alpha_1}e^{2k+2(t-3)(p-1)}.\end{aligned}\\
  \end{gather*} 
  Then 
  \begin{align*}
    \theta_{t-3}\circ(\varphi_{t-2}+\varphi_{t-1})&=-\langle p ,\alpha_{t-3}, \alpha_1 \rangle+\left(\frac{a}{2}+1\right)\langle p ,\alpha_{t-3}, \alpha_2 \rangle\\
    &=-\frac{1}{t-2}\alpha_{t-2}+\frac{a+2}{t-1}\alpha_{t-1}
  \end{align*}
  by Theorem \ref{relation}, \eqref{eq:t-4} and \eqref{Y_k}. Now we let $\theta_{t-1}\colon Y_k^{t-1}\to S^{2k}$ be an extension of $p\circ\theta_{t-3}$ by $\varphi_t$. Then as in the proof of Lemma \ref{composition rank 1}, we obtain that $\theta_{t-1}\vert_{S^{2k}}\neq 0$.
  We apply Lemma \ref{Toda bracket} to the following setting
  \begin{gather*}
 \alpha=\theta_{t-3}\colon Y^{t-3}_k\to S^{2k},\\
  \beta=\varphi_{t-2}+\varphi_{t-1}=-\alpha_1\vee \left(\frac{a}{2}+1\right)\alpha_2\colon S^{2k+2(t-2)(p-1)-1}\vee S^{2k+2(t-1)(p-1)-1}\to Y^{k-3}.\\
 \gamma = \varphi_t\colon S^{2k+2t(p-1)-1}\to Y^{t-3}\cup_{\varphi_{t-2}} e^{2k+2(t-2)(p-1)} \cup_{\varphi_{t-1}} e^{2k+2(t-1)(p-1)}
  \end{gather*} 
 Then
  \begin{align*}
    \theta_{t-1}\circ\varphi_t&=\frac{a+1}{2(t-2)}\langle p ,\alpha_{t-2}, \alpha_2 \rangle+\frac{a+2}{t-1}\langle p ,\alpha_{t-1}, \alpha_1 \rangle\\
    &=\frac{p}{t}\left(\frac{a+1}{t-2}+\frac{a+2}{t-1}\right)\alpha_t'
    =-\frac{p(3a+5)}{2t}\alpha_t'
  \end{align*}
  by Lemma \ref{Toda bracket} and Theorem \ref{relation} together with \eqref{Y_k}.

  \noindent(2) The $t\equiv 1\mod p$ case.

  By Lemma \ref{Toda bracket} and \eqref{Y_k},  we have
  \[
    \theta_{t-3}\circ(\varphi_{t-2}+\varphi_{t-1})=\alpha_{t-2}+\frac{p(a+2)}{t-1}\alpha'_{t-1}.
  \]
  Then as in the case (1), there is a map $\theta_{t-1}\colon Y_k^{t-1}\to S^{2k}$ such that $\theta_{t-1}\vert_{S^{2k}}\ne 0$ and
  \begin{align*}
    \theta_{t-1}\circ\varphi_t&=-\frac{a+1}{2}\langle p^2,\alpha_{t-2}, \alpha_2 \rangle+\frac{p(a+2)}{t-1}\langle p^2,\alpha'_{t-1}, \alpha_1 \rangle\\
    &=\frac{p(a+2)}{t-1}\langle p ,\alpha_{t-1}, \alpha_1 \rangle
    =\frac{p(a+2)}{t-1}\alpha_t
  \end{align*}
  by Lemma \ref{Toda bracket}, Theorem \ref{relation} and \eqref{Y_k}.

  \noindent(3) The $t\equiv 2\mod p$ case.

  By Lemma \ref{Toda bracket} and \eqref{Y_k}, we have
  \[
    \theta_{t-3}\circ(\varphi_{t-2}+\varphi_{t-1})=-\frac{p}{t-2}\alpha'_{t-2}+(a+2)\alpha_{t-1}.
  \]
  Then as in the case (1), there is a map $\theta_{t-1}\colon Y_k^{t-1}\to S^{2k}$ such that $\theta_{t-1}\vert_{S^{2k}}\ne 0$ and
  \begin{align*}
    \theta_{t-1}\circ\varphi_t&=\frac{p(a+1)}{2(t-2)}\langle p^2,\alpha_{t-2}',
    \alpha_2 \rangle+(a+2)\langle p^2 ,\alpha_{t-1}, \alpha_1 \rangle\\
    &=\frac{p(a+1)}{2(t-2)}\langle p ,\alpha_{t-2}, \alpha_2 \rangle
    =\frac{p(a+1)}{2(t-2)}\alpha_t
  \end{align*}
  by Lemma \ref{Toda bracket}, Theorem \ref{relation} and \eqref{Y_k}.

  \noindent(4) The $t\equiv 3\mod p$ with $p>3$ case.

  By Lemma \ref{Toda bracket} and \eqref{Y_k}, we have
  \[
    \theta_{t-3}\circ(\varphi_{t-2}+\varphi_{t-1})=-\alpha_{t-2}+\frac{a+2}{2}\alpha_{t-1}.
  \]
  Then as in the case (1), there is a map $\theta_{t-1}\colon Y_k^{t-1}\to S^{2k}$ such that $\theta_{t-1}\vert_{S^{2k}}\ne 0$ and
  \[
    \theta_{t-1}\circ\varphi_t=\frac{a+1}{2}\langle p,\alpha_{t-2}, \alpha_2 \rangle+\frac{a+2}{2}\langle p,\alpha_{t-1}, \alpha_1 \rangle=\frac{3a+4}{6}\alpha_t
  \]
  by Lemma \ref{Toda bracket}, Theorem \ref{relation} and \eqref{Y_k}.

  \noindent(5) The $t\not\equiv 0,1,2,3\mod p$ case.

  By Lemma \ref{Toda bracket} and \eqref{Y_k}, we have
  \[
    \varphi_{t-3}\circ(\varphi_{t-2}+\varphi_{t-1})=-\frac{1}{t-2}\alpha_{t-2}+\frac{a+2}{t-1}\alpha_{t-1}.
  \]
  Then as in the case (1), there is a map $\theta_{t-1}\colon Y_k^{t-1}\to S^{2k}$ such that $\theta_{t-1}\vert_{S^{2k}}\ne 0$ and
  \[
    \theta_{t-1}\circ\varphi_t=\frac{a+1}{2(t-2)}\langle p
    ,\alpha_{t-2}, \alpha_2 \rangle+\frac{a+2}{t-1}\langle p ,\alpha_{t-1}, \alpha_1 \rangle
    =\frac{1}{t}\left(\frac{a+1}{t-2}+\frac{a+2}{t-1}\right)\alpha_t
  \]
  by Lemma \ref{Toda bracket}, Theorem \ref{relation} and \eqref{Y_k}. Thus the proof is complete.
\end{proof}

\begin{lemma}
  \label{r=3}
  If $k\equiv 1\mod p$ and $3a(k,3)+4\not\equiv 0\mod p$, then there is a map $\theta_2\colon Y_k^2\to S^{2k}$ detecting $\varphi_3$.
\end{lemma}

\begin{proof}
By Lemmas \ref{rank 1} and \ref{rank 4}, we have
  \[
    Y_k^3\simeq S^{2k}\cup_{-\alpha_1}e^{2k+2(p-1)}\cup_{\left(\frac{a}{2}+1\right)\alpha_2}e^{2k+4(p-1)}\cup_{-\frac{a+1}{2}\alpha_2\vee\alpha_1}e^{2k+6(p-1)}
  \]
  where $a=a(k,t)$. Then we can define a map $\theta_2\colon Y_k^2\simeq S^{2k}\cup_{-\alpha_1}e^{2k+2(p-1)}\cup_{\left(\frac{a}{2}+1\right)\alpha_2}e^{2k+4(p-1)}\to S^{2k}$ as an extension of $p\colon S^{2k}\to S^{2k}$ by $-\alpha_1+\left(\frac{a}{2}+1\right)\alpha_2$. Then $\theta_2\vert_{S^{2k}}\ne 0$, and by Lemma \ref{Toda bracket} and Proposition \ref{relation}, the composite $\theta_2\circ\varphi_3$ belongs to the Toda bracket
  \[
   \frac{a+1}{2}\langle p,\alpha_1,\alpha_2\rangle+\left(\frac{a}{2}+1\right)\langle p,\alpha_2,\alpha_1\rangle=
    \begin{cases}
      \frac{3a+4}{2}\alpha_3'&p=3\\
      \frac{3a+4}{6}\alpha_3&p>3.
    \end{cases}
  \]
  Thus $\theta_2$ detects $\varphi_3$, completing the proof.
  \end{proof}

\begin{lemma}
  \label{g 1}
  Given an integer $k,r$ with $1\le r<p(p-1)$, suppose the following conditions:
  \begin{enumerate}
    \item $k(k-r+1)\not\equiv 0\mod p$;

    \item under the above condition, for any integer $3\le t\le r$ satisfying $k-t+2\equiv 0\mod p$, then we further assume
    \begin{alignat*}{3}
      &3a(k,t)+5\not\equiv 0&&\mod p&&(t\equiv 0\mod p)\\
      &a(k,t)+2\not\equiv 0&&\mod p&&(t\equiv 1\mod p)\\
      &3a(k,t)+4\not\equiv 0&&\mod p&&(t\equiv 3\mod p)\\
      &(2t-3)a(k,t)+3t-5\not\equiv 0&&\mod p\qquad&&(t \equiv 4, 5, \cdots , p-1\mod p).
    \end{alignat*}
  \end{enumerate}
  Then there is a map $\theta_{r-1}\colon Y_k^{r-1}\to S^{2k}$ detecting $\varphi_r$.
\end{lemma}

\begin{proof}
  We proceed by induction on $r$ satisfying $k-r+1\not\equiv 0\mod p$. Note that we are considering not all $r$ but satisfying $k-r+1\not\equiv 0\mod p$, for which we can perform induction. For $r=1$, we have $k-r+1=k\not\equiv 0\mod p$ by assumption. Let $\theta_0\colon Y_k^0=S^{2k}\to S^{2k}$ be the identity map of $S^{2k}$. Since $k\not\equiv 0\mod p$, we have $\varphi_1=\alpha_1$ by Lemma \ref{rank 1}. Then we have $\varphi_1^*(\theta_0)=\alpha_1$, and so $\theta_0$ detects $\varphi_1$ by Theorem \ref{homotopy group 2}. For $r=2$, we only need to consider the case $k\not\equiv 1\mod p$ because we are assuming 
  $k-r+1\not\equiv 0 \mod p$. Then by Lemma \ref{composition rank 1}, we get a map $\theta_1$ detecting $\varphi_2$. Suppose $r=3$. If $k\not\equiv 1\mod p$, then we have $\theta_1$ as above, and so by Lemma \ref{composition rank 1}, we get a map $\theta_2$ detecting $\varphi_3$, where we are assuming $k-r+1\not\equiv 0\mod p$. If $k\equiv 1\mod p$, then we can apply Lemma \ref{r=3} to get a map $\theta_2$ detecting $\varphi_3$, 
  where we are assuming $3a(k,3)+4\not\equiv 0\mod p$. 
  Now we assume that for each $4\le t\le r-1$ with $k-t+1\not\equiv 0\mod p$, there is a map $\theta_{t-1}$ detecting $\varphi_t$. If $k-r+2\not\equiv 0\mod p$, then by the induction hypothesis, we have $\theta_{r-2}$, and so by Lemma \ref{composition rank 1} and the assumption $k-r+1\not\equiv 0\mod p$, we get a map $\theta_{r-1}$ detecting $\varphi_r$. If $k-r+2\equiv 0\mod p$, then $k-r+5\not\equiv 0\mod p$, and so by the induction hypothesis, we have $\theta_{r-4}$ detecting $\varphi_{r-3}$. Thus by Lemma \ref{detection}, we also get a map $\theta_{r-1}$ detecting $\varphi_r$, completing the proof.
\end{proof}

\begin{lemma}\label{gen}
If $\theta_{r-2}$
detects
$\varphi_{r-1}$, then 
the extension of $q_{r-1}\circ\theta_{r-2}$ by $\varphi_{r-1}$ is a generator 
in $\pi^{2k}(Y_k^{r-1})/\mathrm{Tor}$.
\end{lemma}

\begin{proof}
The map
$Y_k^{r-1} \to S^{2k}$ is the extension of the restriction of $Y_k^{r-2}$ by $\varphi_{r-1}$.
Consider the exact sequence
  \[
    \pi^{2k}(Y_k^{r-1})\to\pi^{2k}(Y_k^{r-2})\xrightarrow{\varphi_{r-1}^*}\pi^{2k}(S^{2k+2(r-1)(p-1)-1})
  \]
  induced from a cofiber sequence $S^{2k+2(r-1)(p-1)-1}\xrightarrow{\varphi_{r-1}}Y_k^{r-2}\to Y_k^{r-1}$.
Because 
$\theta_{r-2}$ detects $\varphi_{r-1}$, it follows from the exact sequence
that 
if we restrict 
any map $Y_k^{r-1}\to S^{2k}$
on $Y_k^{r-2}$, then 
the restriction is a multiple of 
$q_{r-1}\circ\theta_{r-2}$.
Hence, 
if we extend $q_{r-1}\circ\theta_{r-2}$ by $\varphi_{r-1}$,
then the extension   is a ganerator in 
$\pi^{2k}(Y_k^{r-1})/\mathrm{Tor}$.
\end{proof}

\begin{lemma}
  \label{g 2}
  Suppose that there is a map $\theta_{r-2}\colon Y_k^{r-2}\to S^{2k}$ detecting $\varphi_{r-1}$. Then the map $j_r^*\colon\pi^{2k}(Y_k^r)/\mathrm{Tor}\to \pi^{2k}(Y_k^{r-1})/\mathrm{Tor}$ is an isomorphism whenever $k-r+1\equiv 0\mod p$.
\end{lemma}

\begin{proof}  
  By the assumption and Lemma \ref{gen},   
  an extension $\phi\colon Y_k^{r-1}\to S^{2k}$ of $q_{r-1}\circ\theta_{r-2}$ by $\varphi_{r-1}$ is a generator of $\pi^{2k}(Y_k^{r-1})/\mathrm{Tor}$. By Lemma \ref{Toda bracket}, the composite $\phi\circ\varphi_r$ belongs to the Toda bracket
  \[
    \langle q_{r-1},\theta_{r-2}\circ\varphi_{r-1}, \Sigma^{-1} \rho\circ\varphi_r\rangle,
  \]
  where $\rho\colon Y_k^{r-1}\to S^{2k+2(r-1)(p-1)}$ is the pinch map onto the top cell.
   The indeterminacy of this Toda bracket is
 \[
 q_{r-1}(\pi_{2k+2r(p-1)}(S^{2k}))+ \pi_{2k+2(r-1)(p-1)+1}(S^{2k})\circ  \rho_*(\varphi_r) \qquad (*)
 \]
The first term $q_{r-1}(\pi_{2k+2r(p-1)}(S^{2k}))=0$ vanishes by
Theorem \ref{homotopy group 2}.
We claim $ \rho_*(\varphi_r) =0$. 
  Since $k+(r-2)(p-1)\equiv 1\mod p$ by the assumprtion, Lemma \ref{rank 4} implies that 
  \[
  Y^3_{k+(r-2)(p-1)} = Y^r_k/Y^{r-3}_k = S^{2k+2(r-2)(p-1)}\cup_{-\alpha_1} e^{2k+2(r-1)(p-1)} \cup_{\left(\frac{a}2+1\right)\alpha_2} e^{2k+2r(p-1)}.
  \]   
Note that
  $\varphi_r = (\frac{a}{2}+1) \alpha_2$ and 
  $\alpha_2\colon \partial e^{2k+2r(p-1)} \to S^{2k+2(r-2)(p-1)}$
  is  the attaching map to $S^{2k+2(r-2)(p-1)}$.
  On the other hand, $\rho$ is a map that collapses 
  $S^{2k+2(r-2)(p-1)}$.
  These imply $\rho_*(\varphi_r)=0$.

Hence, $(*)=0$ and the above Toda bracket  consists of a single element. Since $k-r+1\equiv 0\mod p$, we have $\Sigma^{-1} \rho\circ\varphi_r=0$ by Lemma \ref{rank 4}. Then the above Toda bracket includes 0, implying the Toda bracket is trivial. Thus we obtain $\phi\circ\varphi_r=0$.

 Now we consider the exact sequence
  \[
    \pi^{2k}(Y_k^r)\to\pi^{2k}(Y_k^{r-1})\xrightarrow{\varphi_r^*}\pi^{2k}(S^{2k+2r(p-1)-1})
  \]
  induced from a cofiber sequence $S^{2k+2r(p-1)-1}\xrightarrow{\varphi_r^*}Y_k^{r-1}\to Y_r^r$. By the above computation, the map $\varphi_r^*\colon\pi^{2k}(Y_k^{r-1})\to\pi^{2k}(S^{2k+2r(p-1)-1})$ is trivial, implying that the map $j_r^*\colon\pi^{2k}(Y_k^r)/\mathrm{Tor}\to \pi^{2k}(Y_k^{r-1})/\mathrm{Tor}$ is surjective. By Lemma \ref{X-Y}, $\pi^{2k}(Y_k^r)/\mathrm{Tor}\cong\pi^{2k}(Y_k^{r-1})/\mathrm{Tor}\cong\Z_{(p)}$. Thus since any surjection $\Z_{(p)}\to\Z_{(p)}$ is an isomorphism, which contradicts to our assumption. Therefore the proof is finished.
\end{proof}

Now we are ready to prove:

\begin{theorem}
  \label{p odd}
  Given an integer $k,r$ with $1\le r<p(p-1)$, suppose the following conditions:
  \begin{enumerate}
    \item $k\not\equiv 0,1,\ldots,r_p-1\mod p$;

    \item under the above condition, for any integer $3\le t\le r$ satisfying $t-2\equiv k\mod p$, we further assume
    \begin{alignat*}{3}
      &3a(k,t)+5\not\equiv 0&&\mod p&&(t\equiv 0\mod p \text{ and }t\ge p>3)\\
      &a(k,t)+2\not\equiv 0&&\mod p&&(t\equiv 1\mod p\text{ and }t>p)\\
      &3a(k,t)+4\not\equiv 0&&\mod p&&(t\equiv 3\mod p)\\
      &(2t-3)a(k,t)+3t-5\not\equiv 0&&\mod p\qquad&&(t \equiv 4, 5, \cdots , p-1\mod p).
    \end{alignat*}
  \end{enumerate}
  Then the natural map $\pi^{2k}(\C P^{k+r(p-1)})/\mathrm{Tor}\to\pi^{2k}(\C P^k)/\mathrm{Tor}$ is identified with $p^r\colon\Z_{(p)}\to\Z_{(p)}$.
\end{theorem}

\begin{proof}
  By Proposition \ref{reduction}, the map $\pi^{2k}(\C P^{k+r(p-1)})/\mathrm{Tor}\to\pi^{2k}(\C P^k)/\mathrm{Tor}$ is identified with $j_1^*\circ j_2^*\circ\cdots\circ j_r^*\colon\pi^{2k}(Y_k^r)/\mathrm{Tor}\to\pi^{2k}(Y_k^0)/\mathrm{Tor}$. Observe that $r=pq+r_p$ for a non-negative integer $q$ by the definition of $r_p$. Let $0\le s<q$. Then:
  
  \begin{itemize}
  \item
  The map $j_t^*$ is identified with $p\colon\Z_{(p)}\to\Z_{(p)}$ by Theorem \ref{homotopy group 2} and Lemmas \ref{q_t} and \ref{detection}
     for $ps<t<p(s+1)$ with $k-t+1\not\equiv 0\mod p$.
   
   \item 
   The map  $j_{p(s+1)}^*$ is identified with $p^2\colon\Z_{(p)}\to\Z_{(p)}$. 
   
   \item
  There is exactly one $t$ such that $ps<t<p(s+1)$ with $k-t+1\equiv 0\mod p$, for which the map $j_t^*$ is identified with $1\colon\Z_{(p)}\to\Z_{(p)}$ by Lemma \ref{g 2}. 
   
   \end{itemize}
   Then, the composite $j_{ps+1}^*\circ j_{ps+2}^*\circ\cdots\circ j_{p(s+1)}^*$ is identified with $p^p\colon\Z_{(p)}\to\Z_{(p)}$. 
   
   If $pq<t\le r_p$, then $k-t+1\not\equiv 0\mod p$ follows, because
      $k\not\equiv 0,1,\ldots,r_p-1\mod p$. Hence,
     the map $j_t^*$ is identified with $p\colon\Z_{(p)}\to\Z_{(p)}$ by Theorem \ref{homotopy group 2} and Lemmas \ref{q_t} and \ref{detection}. Hence the composite $j_{pq+2}^*\circ j_{pq+3}^*\circ\cdots\circ j_{pq+r_p}^*$ is identified with $p^{r_p}\colon\Z_{(p)}\to\Z_{(p)}$. Thus the composite $j_1^*\circ j_2^*\circ\cdots\circ j_r^*$ is identified with
  \[
    \underbrace{p^p\times\cdots\times p^p}_q\times p^{r_p}=p^{pq+r_p}=p^r\colon\Z_{(p)}\to\Z_{(p)},
  \]
  completing the proof.
\end{proof}

Recall that we have assumed for $p$ to be odd prime at the first paragraph of Section
\ref{Cohomotopy computation 2}.
Below we consider the $p=2$ case.

\begin{proposition}
  \label{p=2}
  If $p=2$ and $k\not\equiv 0\mod p$, then the map $\pi^{2k}(\C P^{k+1})/\mathrm{Tor}\to\pi^{2k}(\C P^k)/\mathrm{Tor}$ is identified with $p\colon\Z_{(p)}\to\Z_{(p)}$.
\end{proposition}

\begin{proof}
  Consider the exact sequence
  \[
    \pi^{2k-1}(\C P^{k-1})\to\pi^{2k}(\C P^n/\C P^{k-1})\to\pi^{2k}(\C P^n)\to\pi^{2k}(\C P^{k-1})
  \]
  induced from the homotopy cofibration $\C P^{k-1}\to\C P^n\to\C P^n/\C P^{k-1}$ for $n\ge k$. Since $\C P^{k-1}$ is of dimension $2k-2$, we have $\pi^{2k-1}(\C P^{k-1})=\pi^{2k}(\C P^{k-1})=0$, and so the natural map $\pi^{2k}(\C P^n/\C P^{k-1})\to\pi^{2k}(\C P^n)$ is an isomorphism. Note that the inclusion $\C P^k\to\C P^{k+1}$ induces a commutative diagram
  \[
    \xymatrix{
      \pi^{2k}(\C P^{k+1}/\C P^{k-1})\ar[r]^(.57)\cong\ar[d]&\pi^{2k}(\C P^{k+1})\ar[d]\\
      \pi^{2k}(\C P^k/\C P^{k-1})\ar[r]^(.57)\cong&\pi^{2k}(\C P^k).
    }
  \]
  Then the map $\pi^{2k}(\C P^{k+1})/\mathrm{Tor}\to\pi^{2k}(\C P^k)/\mathrm{Tor}$ is identified with the map
  \[
    \pi^{2k}(\C P^{k+1}/\C P^{k-1})/\mathrm{Tor}\to\pi^{2k}(\C P^k/\C P^{k-1})/\mathrm{Tor}.
  \]

  Clearly, $\C P^k/\C P^{k-1}\cong S^{2k}$ holds. Because $k\not\equiv 0 \mod 2$, it is well known that $\C P^{k+1}\simeq S^{2k}\cup_\eta e^{2k+2}$ such that the inclusion $\C P^k/\C P^{k-1}\to \C P^{k+1}/\C P^{k-1}$ is identified with the bottom cell inclusion, where $\eta$ is a generator of $\pi_{2k+1}(S^{2k})\cong\Z/2$. Consider the exact sequence
  \[
    \pi^{2k}(\C P^{k+1}/\C P^{k-1})\to\pi^{2k}(\C P^k/\C P^{k-1})\to\pi^{2k}(S^{2k+1})
  \]
  induced from the cofiber sequence $S^{2k+1}\to\C P^k/\C P^{k-1}\to\C P^{k+1}/\C P^{k-1}$. By the above observation, this exact sequence is identified with the exact sequence
  \[
    \pi^{2k}(S^{2k}\cup_\eta e^{2k+2})\to\pi^{2k}(S^{2k})\xrightarrow{\eta^*}\pi^{2k}(S^{2k+1})=\pi_{2k+1}(S^{2k})
  \]
  induced from the cofiber sequence $S^{2k+1}\xrightarrow{\eta}S^{2k}\to S^{2k}\cup_\eta e^{2k+2}$. Since $\eta^*(1)=\eta$ and $\pi_{2k+1}(S^{2k})$ is generated by $\eta$, the second map is surjective. Then the map $\pi^{2k}(\C P^{k+1}/\C P^{k-1})/\mathrm{Tor}\to\pi^{2k}(\C P^k/\C P^{k-1})/\mathrm{Tor}$ is identified with $p\colon\Z_{(p)}\to\Z_{(p)}$ with $p=2$, completing the proof.
\end{proof}

Finally, we prove Theorem \ref{cohomotopy CP'}.

\begin{proof}
  [Proof of Theorem \ref{cohomotopy CP'}]
  Combine Theorem \ref{p odd} and Proposition \ref{p=2} below.
\end{proof}


\section{Inexplicit upper bound}\label{inexplicit upper bound}

This section explains an upper bound mentioned in Section \ref{introduction}. So it is completely independent from other sections, and does not contain any result. Let $X$ be a closed, connected, oriented and smooth four-manifold, and let $\mathfrak{s}$ be a $\spin^c$ structure on $X$. We recall a result of Bauer and Furuta \cite[Theorem 3.7]{BF}.

\begin{theorem}
  \label{divisibility}
  If $b_2^+\ge 2$ and $b_1=0$, then $SW(\mathfrak{s})$ is divisible by the denominator of $a_i^{(k)}$ for $1\le i\le d(\mathfrak{s})/2$, where $k=(b_2^+-1)/2$ and $a_i^{(k)}$ is defined by
  \[
    \left(-\frac{\log(1-x)}{x}\right)^k=\left(1+\frac{x}{2}+\frac{x^2}{3}+\cdots+\frac{x^{n-1}}{n}+\cdots\right)^k=1+\sum_{i\ge 1}a_i^{(k)}x^i.
  \]
\end{theorem}

Let $d(q,k)$ denote the greatest integer $2d$ such that the denominator of $a_i^{(k)}$ is not divisible by $q$ for $1\le i\le d$. By Theorem \ref{divisibility}, we get that if $\mathfrak{s}$ is a mod $p$ basic class for a prime $p$ and $k=(b_2^+-1)/2$, then there is an inequality
\[
  d(\mathfrak{s})\le d(p,k).
\]
By putting $x=e^y-1$, we can see that the numbers $a_i^{(k)}$ are computed from the Bernoulli numbers, and vice versa. Then it is quite hard to determine or evaluate $a_i^{(k)}$, in general. Thus the upper bound $d(p,k)$ is rather inexplicit, in general. However, we can compute $d(p,q)$ in the following two special cases. First, we clearly have $d(q,1)=2q-4$. Then as mentioned in Section \ref{introduction}, for a prime $p$ and $k=1$, our upper bound $2r(p-1)-2$ in Theorem \ref{main2} is much sharper than the upper bound $d(p^r,1)=2p^r-4$, except for a few cases. Second, we let for an integer $1\le r<p$. Then since $r(p-1)<p^2+1$, in the expansion of
\[
  \left(1+\frac{x}{2}+\frac{x^2}{3}+\cdots+\frac{x^{n-1}}{n}+\cdots\right)^k,
\]
the coefficient of $x^{r(p-1)}$ is
\[
  \frac{1}{\lambda_1p^{r_1}}+\cdots+\frac{1}{\lambda_np^{r_n}}+\binom{k}{r}\frac{1}{p^r}=\frac{p(\lambda_1p^{r-r_1}+\cdots+\lambda_np^{r-r_n})+\lambda_1\cdots\lambda_n\binom{k}{r}}{\lambda_1\cdots\lambda_n p^r}
\]
where $\lambda_1,\ldots,\lambda_n\not\equiv 0\mod p$ and $r_1,\ldots,r_n<r$. If $k\not\equiv 0,1,\ldots,r-1\mod p$, then $\binom{k}{r}\not\equiv 0\mod p$, and so the numerator is not divisible by $p$. Hence $d(p^r,k)\ge 2r(p-1)-2$. On the other hand, we can see that the denominator of the coefficient of $x^i$ for $i<r(p-1)$ is not divisible by $p^r$ quite similarly. Then we get $d(p^r,k)\ge 2r(p-1)-2$, hence $d(p^r,k)=2r(p-1)-2$. This gives an alternative proof of Corollary \ref{main3}, where the existence of an integer $k$ satisfying $k\not\equiv 0,1,\ldots,r-1\mod p$ implies $r<p$.


\begin{thebibliography}{99}
  \bibitem{AM97}S. Akbulut and R. Matveyev, Exotic structures and adjunction inequality, Turkish J. Math. \textbf{21} (1997), no. 1, 47-53.

  \bibitem{AY13}  S. Akbulut and K. Yasui, Cork twisting exotic Stein 4-manifolds, J. Differential Geom. \textbf{93} (2013), no. 1, 1-36.

  \bibitem{ABMRS} N. Amersi, O. Beckwith, S.J. Miller, R. Ronan, and J. Sondow, Generalized Ramanujan primes. (English summary) Combinatorial and additive number theory-CANT 2011 and 2012, 1-13, Springer Proc. Math. Stat. \textbf{101}, Springer, New York, 2014.

  \bibitem{BF} S. Bauer and M. Furuta, A stable cohomotopy refinement of Seiberg-Witten invariants. I, Invent. Math. \textbf{155} (2004), no. 1, 1-19.

  \bibitem{B} S. Bauer, A stable cohomotopy refinement of Seiberg-Witten invariants. II, Invent. Math. \textbf{155} (2004), no. 1, 21-40.


  \bibitem{D} D.M. Davis, Stable $p$-equivalences of stunted complex projective spaces, Indiana Univ. Math. J. \textbf{28} (1979), no. 1, 23-34.

  \bibitem{FL15}P. Feehan and T.G. Leness, Witten's Conjecture for many four-manifolds of simple type, J. Eur. Math. Soc. \textbf{17} (2015), no.\ 4, 899--923.


  \bibitem{FS} R. Fintushel and R.J. Stern, Immersed spheres in 4-manifolds and the immersed Thom conjecture, Turk. J. Math. \textbf{19} (1995), no. 2, 145-157.

  \bibitem{G17} R.E. Gompf, Minimal genera of open 4-manifolds, Geom. Topol. \textbf{21} (2017), no. 1, 107--155.

  \bibitem{HMR} P. Hilton, G. Mislin, and J. Roitberg, Localization of Nilpotent Groups and Spaces, North Holland, Math. Studies \textbf{15}, 1975.

  \bibitem{IS} M. Ishida and H. Sasahira, Stable cohomotopy Seiberg-Witten invariants of connected sums of four-manifolds with positive first Betti number, I: non-vanishing theorem, Internat. J. Math. \textbf{26} (2015), no. 6, 1541004.


  \bibitem{KNY} T. Kato, N. Nakamura, and K. Yasui, The simple type conjecture for mod 2 Seiberg-Witten invariants, accepted by J. Eur. Math. Soc.

  \bibitem{KM94}  P. Kronheimer and T. Mrowka, The genus of embedded surfaces in the projective plane, Math. Res. Lett. \textbf{1} (1994), no. 6, 797--808.

  \bibitem{KM}  P. Kronheimer and T. Mrowka, Monopoles and Three-manifolds, New Mathematical Monographs \textbf{10}, Cambridge University Press, 2007.


  \bibitem{MNT} M. Mimura, G. Nishida, and H. Toda, Localization of $CW$-complexes and its applications, J. Math. Soc. Japan \textbf{23} (1971), 593-624.


  \bibitem{MST}  J.W. Morgan, Z. Szab\'{o}, and C.H. Taubes, A product formula for the Seiberg-Witten invariants and the generalized Thom conjecture, J. Differential Geom. \textbf{44} (1996), no. 4, 706-788.

  \bibitem{N} L.I. Nicolaescu, Notes on Seiberg-Witten Theory, Graduate Studies in Mathematics \textbf{28}, American Mathematical Society, Providence, RI, 2000.

  \bibitem{O} H. \={O}shima, A problem on coextension, Math. J. Ibaraki \textbf{51} (2019), 27-38.

  \bibitem{OS1} P. Ozsv\'{a}th and Z. Szab\'{o}, The symplectic Thom conjecture, Ann. Math. (2) \textbf{151} (2000), no. 1, 93-124.

  \bibitem{OS2} P. Ozsv\'{a}th and Z. Szab\'{o}, Higher type adjunction inequalities in Seiberg-Witten theory, J. Differ. Geom. \textbf{55} (2000), no. 3, 385-440.

  \bibitem{RS} J.B. Rosser and L. Schoenfeld, Approximate formulas for some functions of prime numbers, Illinois J. Math. \textbf{6} (1962), 64-94.

  \bibitem{S} J. Sondow, Ramanujan primes and Bertrand's postulate, Amer. Math. Monthly \textbf{116} (2009), no. 7, 630-635.

  \bibitem{Ta}  C.H. Taubes, SW $\Rightarrow$ Gr: from the Seiberg-Witten equations to pseudo-holomorphic curves, J. Amer. Math. Soc. \textbf{9} (1996), no. 3, 845-918.

  \bibitem{T1} H. Toda, $p$-primary components of homotopy groups IV. Compositions and toric constructions, Memoirs Univ. Kyoto \textbf{32} (1959), 297-332.

  \bibitem{T2} H. Toda, Composition Methods in Homotopy Groups of Spheres, Annals of Mathematics Studies \textbf{49}, Princeton University Press, Princeton, N.J., 1962


  \bibitem{Y19GT}  K. Yasui, Geometrically simply connected 4-manifolds and stable cohomotopy Seiberg-Witten invariants, Geom. Topol. \textbf{23} (2019), no. 5, 2685-2697.
\end{thebibliography}
\end{document}